\documentclass[a4paper,11pt]{article}
\usepackage{amsmath}
\usepackage{amssymb}

\usepackage{setspace}

\usepackage{amsthm}
\newtheorem{theorem}{Theorem}[section]
\newtheorem{lemma}[theorem]{Lemma}

\newtheorem{result}[theorem]{Result}

\theoremstyle{definition}

\newtheorem*{definition}{Definition}

\usepackage{pdfsync}
\def\PG{\mathrm{PG}} \def\AG{\mathrm{AG}} 
  \def\Persp{\mathrm{Persp}}
\def\Aut{\mathrm{Aut}}
\def\PGammaL{\mathrm{P}\Gamma\mathrm{L}}
\def\PGL{\mathrm{PGL}} 
\def\GL{\mathrm{GL}}

\def\A{\mathcal{A}} \def\B{\mathcal{B}} 
\def\D{\mathcal{D}}  
 \def\K{\mathcal{K}} \def\FF{\mathcal{F}}

\def\L{\mathcal{L}} \def\M{\mathcal{M}} 
 \def\P{\mathcal{P}}
 \def\S{\mathcal{S}}

\def\F{\mathbb{F}}

\def\a{\alpha}

\title{Linear representations of subgeometries}
\author{Stefaan De Winter\thanks{Supported by Michigan Technological University REF grant R01289} \and Sara Rottey\thanks{Partially supported by Michigan Technological University REF grant R01289} \and Geertrui Van de Voorde\thanks{Supported by the Fund for Scientific Research -- Flanders (FWO).}}
\begin{document}
\maketitle
\begin{abstract}The linear representation $T_n^*(\K)$ of a point set $\K$ in a hyperplane of $\PG(n+1,q)$ is a point-line geometry embedded in this projective space. In this paper, we will determine the isomorphisms between two linear representations $T_n^*(\K)$ and $T_n^*(\K')$, under a few conditions on $\K$ and $\K'$.
%We introduce the notion of the closure $\overline{\K}$ of a point set $\K$, which is the smallest subgeometry in which $\K$ is contained.
First, we prove that an isomorphism between $T_n^*(\K)$ and $T_n^*(\K')$ is induced by an isomorphism between the two linear representations $T_n^*(\overline{\K})$ and $T_n^*(\overline{\K'})$ of their closures $\overline {\K}$ and $\overline{\K'}$.

This allows us to focus on the automorphism group of a linear representation $T_n^*(\S)$ of a subgeometry $\S\cong\PG(n,q)$ embedded in a hyperplane of the projective space $\PG(n+1,q^t)$.
To this end we introduce a geometry $X(n,t,q)$ and determine its automorphism group. The geometry $X(n,t,q)$ is a straightforward generalization of $H_{q}^{n+2}$ which is known to be isomorphic to the linear representation of a Baer subgeometry. By providing an elegant algebraic description of $X(n,t,q)$ as a coset geometry we extend this result and prove that $X(n,t,q)$ and $T_n^*(\S)$ are isomorphic. %Hence, we can conclude their automorphism groups are isomorphic.

Finally, we compare the full automorphism group of $T^*_n(\S)$ with the ``natural'' group of automorphisms that is induced by the collineation group of its ambient space. %This allows to correct a misconception that has appeared in the literature.
\end{abstract}
{\bf Keywords:} Linear representation, automorphism group, subgeometry, coset geometry

\noindent
{\bf MSC:} 51E20

\section{Introduction}\label{introduction}
In finite geometry, one often considers geometries that are embedded in a projective or affine space. If  $G_1$ and $G_2$ are two such geometries, embedded in the same space, we may ask the following question:

\begin{itemize}\item[(Q)]
Is every isomorphism between $G_1$ and $G_2$ induced by a collineation of the ambient space?
\end{itemize}

Of course, the answer to this question will strongly depend on the properties of the geometries $G_1$ and $G_2$, as well as on the type of embedding considered. For example, it is well known that this question has an affirmative answer for the standard embeddings of finite classical polar spaces and Segre varieties \cite{GGG}, and  has  been studied for various other types of geometries such as generalised quadrangles \cite{PayneThas} and semipartial geometries \cite{StefaanTSPG}.

In this paper, we study the problem (Q) for a particular but broad class of geometries, namely {\em linear representations}.

The  linear representation $T_2^*(\mathcal{O})$, where $\mathcal{O}$ is a hyperoval, was introduced by Ahrens and Szekeres \cite{AS} and independently by Hall \cite{MH}.  This definition was extended to the linear representation of general point sets in  \cite{DeClerck}.
\begin{definition} Consider the projective space $\PG(n+1,q)$ for some prime power $q$. Let $H_\infty$ be a hyperplane in $\PG(n+1,q)$ and let $\K$ be a point set in $H_\infty$. The {\em linear representation} $T_n^*(\K)$  of the point set $\K$ is the point-line incidence structure $(\P, \L)$ whose points and lines are as follows:
\begin{itemize}%\setlength{\itemindent}{3.5em}
\item[$\P:$] the {\em affine} points of $\PG(n+1,q)$, i.e. the points of $\PG(n+1,q)\setminus H_\infty$,
\item[$\L:$] the lines of $\PG(n+1,q)$ intersecting $H_\infty$ exactly in a point of $\K$.
\end{itemize}
Incidence is natural.
\end{definition}

We see that a linear representation $T_n^*(\K)$ in $\PG(n+1,q)$ is entirely determined by the point set $\K$ at infinity.

Linear representations are mostly studied for point sets $\K$ that possess a lot of symmetry. For example, in the case $n=2$ and $\K$ a hyperoval, $T_2^*(\K)$ is a generalised quadrangle of order $(q-1,q+1)$.  Bichara, Mazzocca and Somma showed in \cite{bichara} that  for $\K, \K'$ hyperovals, $T_2^*(\K)\cong T_2^*(\K')$ if and only if $\K$ and $\K'$ are $\PGammaL$-equivalent. The answer to (Q) is given in the affirmative when $\K$ is a regular hyperoval in $\PG(2,q)$ in \cite{Grund}. When $\K$ is a Buekenhout-Metz unital, question (Q) is answered by De Winter in \cite{stefaan}; in this case, the linear representation $T_2^*(\K)$ is a {\em semipartial geometry}.
In \cite{wij2} the authors proved that the answer to (Q) for two linear representations $T_n^*(\K)$ and $T_n^*(\K')$ is affirmative provided certain conditions on $\K$ are fulfilled (see Result \ref{SGP}). When one of these conditions is not met, the automorphism group can be larger. % It is worth noticing that this last result includes all previously mentioned cases.
The same conditions might have been silently assumed when the authors of \cite[Corollary 7]{GuglielmoDesSpreads} answer the question (Q) affirmatively in the case of (generalised) linear representations; since in \cite{wij2} the authors have shown that the answer to question (Q) is in general negative for linear representations.

A {\em frame} of $\PG(n,q)$ is a set of $n+2$ points where each subset of $n+1$ points spans $\PG(n,q)$.
When a point set $\K$ of $\PG(n,q)$ contains a frame, the {\em closure} $\overline{\K}$ of $\K$ consists of the points of the smallest subgeometry of $\PG(n,q)$ containing $\K$. The case $\overline{\K}=H_\infty$ was treated in \cite{wij2}. This article is concerned with finding the full automorphism group of $T_n^*(\K)$ in the case $\overline{\K} \neq H_\infty$.

In Section \ref{sectionsubsets} we prove that an isomorphism between $T^*_n(\K)$ and $T^*_n(\K')$ is induced by an isomorphism between $T^*_n(\overline{\K})$ and $T^*_n(\overline{\K'})$. This explains why the rest of the paper is devoted to isomorphisms between $T^*_n(\S)$ and $T^*_n(\S')$ for subgeometries $\S \cong \S'$ of $H_\infty$.

In Section \ref{stabgeom} we introduce the geometry $X(n,t,q)$, which is a generalisation of the semipartial geometry $H_q^{n+2}$ that was introduced in \cite{DeClerck}. We will explore the automorphism group of $X(n,t,q)$ and prove that this group consists solely of collineations of its ambient space $\PG(n+t+1,q)$.

In Section \ref{sectionstefaan} we study $X(n,t,q)$ as a coset geometry, generalising the idea of translation semipartial geometries as introduced in \cite{StefaanTSPG}. In a group theoretical / algebraic way, we recover the isomorphism with $T_n^*(\S)$, where $\S$ is a subgeometry $\PG(n,q)$ of the hyperplane $H_\infty\cong \PG(n,q^t)$. This yields a useful
and elegant algebraic description of $T_n^*(\S)$ and its full automorphism group. It, for example, allows an explicit description of the automorphism group making it possible to perform computations in this group.

In Section \ref{sectionblowup} we prove that the automorphism group of $T_n^*(\S)$ is isomorphic to a specific collineation group of $\PG(t(n+1)+1,q)$, namely the collineation group stabilising the generalised linear representation isomorphic to $T_n^*(\S)$. These results in particular show that the automorphism group of $T_n^*(\S)$ is much larger than its automorphism group induced by $\PGammaL(n+2,q)$. We end with a full answer to the isomorphism problem of linear representations $T_n^*(\K)$ and $T_n^*(\K')$ of point sets $\K$, $\K'$ in $\PG(n,q^t)$, such that the closure of $\K,\K'$ is a subgeometry $\PG(n,q)$.

%This corrects the misconception that this induced automorphism group is the full automorphism group.
\section{Linear representations}\label{sectionsubsets}

First of all, we wish to emphasize the distinction between a {\em subspace} and a {\em subgeometry}. A subspace of $\PG(n,q)$ is a projective space $\PG(m,q)$ contained in $\PG(n,q)$, $m \leq n$, over the same finite field $\F_q$.
%An $n$-dimensional subgeometry of $\PG(n,q)$ is a projective space $\PG(n,q_0)$ contained in $\PG(n,q)$ over some subfield $\F_{q_0}$ of $\F_q$.
Now consider a subfield $\F_{q_0^{}}$ of $\F_q$; an $n$-dimensional subgeometry of $\PG(n,q)$ of order $q_{0}^{}$ is a set of $(q_0^{n+1}-1)/(q_0^{}-1)$ points whose homogeneous coordinates, with respect to a fixed frame of $\PG(n,q)$, are in $\F_{q_0^{}}$.

%\begin{definition} Consider the projective space $\PG(n+1,q)$ for some prime power $q$. Let $H_\infty=\PG(n,q)$ be a fixed hyperplane in $\PG(n+1,q)$ and let $\K$ be a point set in $H_\infty$. The {\em linear representation} $T_n^*(\K)$  of the point set $\K$ is the incidence structure $(\P, \L)$ as follows:
%\begin{itemize}%\setlength{\itemindent}{3.5em}
%\item[$\P:$] the {\em affine} points of $\PG(n+1,q)$, i.e. the points of $\PG(n+1,q)\setminus H_\infty$,
%\item[$\L:$] the lines of $\PG(n+1,q)$ intersecting $H_\infty$ in exactly a point of $\K$.
%\end{itemize}
%The incidence is natural.
%\end{definition}

In this paper we will consider a linear representation $T_n^*(\K)$ of a point set $\K$ contained in a hyperplane $H_\infty$, containing a frame and satisfying specific conditions.
In order to state these conditions we need the following two definitions.

\begin{definition}
We say a point set $\K$ of $\PG(n,q)$ has {\em  Property (\texttt{*})} if there is no plane of $\PG(n,q)$ intersecting $\K$ in only two intersecting lines or in two intersecting lines minus their intersection point.
\end{definition}

We will exclude all point sets that do not possess {\em  Property (\texttt{*})}. Note that in \cite{wij2}, considering a linear representation of a specific point set not having this property, an explicit example of an automorphism that is not a collineation of the ambient space is given.
%eigenlijk sterker, namelijk niet F_p- lineair?

\begin{definition}
If a point set $S$ of $\PG(n,q)$ contains a frame of $\PG(n,q)$, then its {\em closure} $\overline{S}$ consists of the points of the
smallest $n$-dimensional subgeometry of $\PG(n,q)$ containing all points of $S$.

The closure $\overline{S}$ of a point set $S$ can be constructed recursively as follows:
\begin{itemize}
\item[(i)] determine the set $\A$ of all subspaces of $\PG(n,q)$ spanned by an arbitrary number of points of $S$;
\item[(ii)] determine the set $\overline{S}$ of all points that arise as the exact intersection of any
two subspaces in $\A$, that is, $P\in \overline{S}$ if and only there exist two subspaces $\pi_1$ and $\pi_2$ spanned by points of $S$, such that $P=\pi_1\cap\pi_2$. If $\overline{S} \neq S$ replace $S$ by
$\overline{S}$ and go to~(i), otherwise stop.
\end{itemize}
\end{definition}
For $n=2$, this recursive construction coincides with the definition of the closure of a set of points in a plane containing a quadrangle, given in \cite[Chapter XI]{Hughes}.

The result in \cite{wij2} concerning the isomorphisms between linear representations $T_n^*(\K)$ and $T_n^*(\K')$ for point sets $\K$ and $\K'$, is the following.

\begin{result}{\rm \label{SGP}\cite{wij2}} Let $q>2$. Let $\K$ and $\K'$ denote point sets spanning $H_\infty\cong\PG(n,q)$, having Property (\texttt{*}) in $H_\infty$, such that the closure $\overline{\K}$ is equal to $H_\infty$. Let $\alpha$ be an isomorphism between $T_n^*(\K)$ and $T_n^*(\K')$.
Then $\alpha$ is induced by an element of the stabiliser $\PGammaL(n+2,q)_{H_\infty}$ of $H_\infty$ mapping $\K$ to $\K'$.
\end{result}

%We will now suppose that our set $\K$ does not meet the condition on its closure $\overline{\K}$. Consider a point set $\K$ of $H_\infty=\PG(n,q^t)$, $t\neq1$, such that $\overline{\K}=\S=\PG(n,q)$, i.e. is a subgeometry of $\PG(n,q^t)$.
We will prove that, if $\K$ or $\K'$ satisfies Property (\texttt{*}), an isomorphism between the linear representations $T_n^*(\K)$ and $T_n^*(\K')$ of two point sets $\K$ and $\K'$ is always induced by an isomorphism between the linear representations $T_n^*(\overline{\K})$ and $T_n^*(\overline{\K'})$ of their closures $\overline{\K}$ and $\overline{\K'}$.

We need the notion of {\em $\alpha$-rigid} subspaces (see \cite{wij2}):

\begin{definition} Let $\alpha$ be an isomorphism between $T_n^*(\K)$ and $T_n^*(\K')$.  We
will say that a $k$-subspace $\pi_\infty$ of $H_\infty$ is \emph{$\a$-rigid}
if for every $(k+1)$-subspace $\pi$ containing
$\pi_\infty$, but not contained in $H_\infty$, the point set $\{ \a(P)|P
\in \pi, P\notin H_\infty\}$ spans a $(k+1)$-subspace.
\end{definition}

It follows from the definition of a linear representation $T_n^*(\K)$ that every point of
$\K$ is $\beta$-rigid, for any isomorphism $\beta$.

We will need the following theorems of \cite{wij2}.

\begin{result}{\rm \label{intersectie}\cite{wij2}} If for some isomorphism $\a$, the subspaces $\pi_1$ and $\pi_2$ are $\a$-rigid
subspaces meeting in at least one point, then $\pi_1\cap \pi_2$ is an $\a$-rigid
subspace.
\end{result}
%\begin{result} \label{help}Let $q>2$. Suppose $\K$ or $\K'$ has Property ($*$). Let $P_1$ and $P_2$ be two points of $\K$, then $P_1P_2$ is rigid.\end{result}
For a line $L$ of $\PG(n+1,q)$ not in $H_\infty$, we define $\infty(L)$ to be the point $L\cap H_\infty$.

  An isomorphism $\beta$ of the affine points extends naturally to a mapping on the lines $L$ having a $\beta$-rigid point $\infty(L)$ at infinity, by defining $\beta(L)$ to be the line containing the points $\beta(R)$, $R\in L$.

\begin{result}{\rm \label{crucial}\cite{wij2}} Let $q>2$. Suppose $\K$ or $\K'$ has Property ($*$) and $\langle \K \rangle = \langle \K' \rangle = H_\infty$. Let $\a$ be an isomorphism between $T_n^*(\K)$ and $T_n^*(\K')$. We can define a mapping $\tilde{\a}$ on the set of $\a$-rigid points by putting $\tilde{\a}(Q)=\infty(\a(L))$ where $Q$ is an $\a$-rigid point and $L$ is any line for which $\infty(L)=Q$. This means, for two lines $L$ and $M$, if $\infty(L)=\infty(M)=Q$ is an $\a$-rigid point, then $\infty(\a(L))=\infty(\a(M))=\tilde{\a}(Q)$.\end{result}
This result shows that an isomorphism $\beta$ between $T_n^*(\K)$ and $T_n^*(\K')$ can be extended to a mapping on the $\beta$-rigid points $Q \in H_\infty$, and we abuse notation by putting $\beta(Q):=\tilde{\beta}(Q)$.
\begin{result}{\rm\label{opspanning}\cite{wij2}} Let $q>2$. Suppose $\K$ or $\K'$ has Property ($*$) and $\langle \K \rangle = \langle \K' \rangle = H_\infty$. Let $\a$ be an isomorphism between $T_n^*(\K)$ and $T_n^*(\K')$. If $P_1,\ldots,P_{k+1}$ are $\a$-rigid
points, then $\langle P_1,\ldots,P_{k+1}\rangle$ is an $\a$-rigid space.
\end{result}

In the following theorem we generalise Result \ref{SGP} for point sets $\K$ whose closure is not necessarily $H_\infty$.

\begin{theorem}\label{mainmap}
Let $\K$ and  $\K'$ be point sets in the hyperplane $H_\infty\cong\PG(n,q)$ such that $\S=\overline{\K}$ and $\S'=\overline{\K'}$ are $n$-dimensional subgeometries of $H_\infty$. If $\S=H_\infty$, then further suppose $\K$ satisfies Property $(*)$.
An isomorphism $\gamma$ between $T_n^*(\K)$ and $T_n^*(\K')$ is induced by an isomorphism between $T_n^*(\S)$ and $T_n^*(\S')$ that maps $\S$ onto $\S'$.

\end{theorem}
\begin{proof}
If the subgeometry $\S$ is not the whole hyperplane $H_\infty$, then the set $\K$ spans $H_\infty$, but it does not contain full lines of $H_\infty$, nor full lines minus one point, and thus it also has Property ($*$). By the recursive construction of the closure of a set of points, we conclude, invoking Result \ref{intersectie} and Result \ref{opspanning}, that all points of $\S$ are $\gamma$-rigid. Hence $\gamma$ maps the affine points of a projective line intersecting $H_\infty$ in exactly one point of $\S$, onto the affine points of a projective line  intersecting $H_\infty$ in exactly one point.

Making use of Result \ref{crucial}, we know that lines through a $\gamma$-rigid point at infinity are mapped to lines intersecting each other in a point at infinity. As said before, we abuse notation and use $\gamma$ for the extension of $\gamma$ to all $\gamma$-rigid points of $H_\infty$. Now, let $P,Q,R$ be three points of $\S$ on a line $L$ of $H_\infty$ and let $U$ be a point, not contained in $H_\infty$. Since $L$ is a $\gamma$-rigid line, $\gamma$ maps the points of $\langle U,L\rangle$ onto points of a plane containing $\gamma(P), \gamma(Q)$, and $\gamma(R)$ at infinity. This implies that $\gamma$ also maps collinear points of $\S$ to collinear points of $H_\infty$. Since $\K$ is mapped to $\K'$, and collinearity needs to be preserved, clearly the points of $\S$ are mapped to the points of $\S'$ (keeping the recursive construction of $\S$ and $\S'$ in mind).

With the same argument, the points of $\S'$ are $\gamma^{-1}$-rigid, and collinear points of $\S'$ are mapped by $\gamma^{-1}$ to collinear points of $H_\infty$, thus belonging to $\S$. We conclude that $\gamma$ is induced by an isomorphism between $T_n^*(\S)$ and $T_n^*(\S')$ that maps $\S$ onto $\S'$ (preserving collinearity of points of $\S$).
\end{proof}

The automorphism group of $T_n^*(\K)$ where $\overline{\K}=H_\infty$ has been determined in \cite{wij2}.
It follows from Theorem \ref{mainmap} that we only need to consider the automorphism group of $T_n^*(\S)$ where $\S$ is a subgeometry of $H_\infty$.
%There the authors proved that $\gamma$ is induced by an element of $\PGammaL(n+1,q)_{H_\infty}$. In the following sections we are concerned with the case $\S\neq H_\infty$. By Theorem \ref{mainmap} we see that $\S\cong\S'$, hence we only need to find the automorphism group of $T_n^*(\S)$. We will prove that in this case there exist automorphisms that are not induced by an element of $\PGammaL(n+1,q)_{H_\infty}$.
%
%So, from now on we will only consider the linear representation of a non-trivial $n$-dimensional subgeometry of $H_\infty$. Let $\S\cong\PG(n,q)$ be a subgeometry of $H_\infty\cong\PG(n,q^t)$, $q$ a prime power, $t \in \N_0 \backslash \{1\}$, and embed $H_\infty$ in $\PG(n+1,q^t)$; consider the linear representation $T^*_{n}(\S)$ of $\S$.
%\begin{remark}
%This case covers all possibilities of non-trivial subgeometries, since a subgeometry $\PG(n,q^i)$ of $\PG(n,q^t)$ can be written as the subgeometry $\PG(n,\tilde{q})$ of $\PG(n,\tilde{q}^{\frac{t}{i}})$ where $\tilde{q}=q^i$ since $i|t$.
%\end{remark}

\section{The geometry $X(n,t,q)$ and its automorphism group}\label{stabgeom}
\begin{definition} Consider an $n$-dimensional subspace $\pi$ of the projective space $\PG(n+t,q)$. The geometry $X(n,t,q)$ is the incidence structure $(\P, \L)$ with natural incidence and with
\begin{itemize}
\item[$\P$:] the $(t-1)$-spaces of $\PG(n+t,q)$ skew to $\pi$,
\item[$\L$:] the $t$-spaces of $\PG(n+t,q)$ meeting $\pi$ in exactly one point.
\end{itemize}
\end{definition}

For $t=2$, this geometry was introduced in \cite{DeClerck} as $H^{n+2}_{q}$. The geometry $H^{n+2}_{q}$ was also given in \cite{Debroey} as an example of a semipartial geometry that is not a partial geometry for $n\geq2$. There the author proved that for $n=2$ the geometry $H_q^4$ is isomorphic to $T_2^*(\B)$ where $\B$ is a Baer subplane of $\PG(2,q^2)$. In Section \ref{sectionstefaan} we will prove this for general $n$ and $t$, namely $X(n,t,q) \cong T^*_n(\S)$, where $\S$ is an $n$-dimensional subgeometry over $\F_q$ of $\PG(n,q^t)$.

%Two linear representations $T_n^*(\S)$ and  $T_n^*(\S')$ embedded in $\PG(n,q^t)$ for which $\S\cong\S'\cong\PG(n,q)$ are both isomorphic to the geometry $X(n,t,q)$. Hence, any isomorphism between $T_n^*(\S)$ and  $T_n^*(\S')$ corresponds to an automorphism of $X(n,t,q)$.
First, we will determine the full automorphism group of $X(n,t,q)$.

Consider the embedding of the geometry $X(n,t,q)$ in the projective space $\PG(n+t,q)$ where $\pi$ is an $n$-dimensional subspace. We will prove that an automorphism of $X(n,t,q)$ %extends to an incidence preserving mapping on the points of $\PG(n+t,q)$ and hence 
is induced by a collineation of its ambient space.

\begin{lemma}
All automorphisms of $X(n,t,q)$ are induced by collineations of its ambient space; more precisely, $\Aut(X(n,t,q))\cong \PGammaL(n+t+1,q)_{\pi}$.
\end{lemma}
\begin{proof}
We will prove that every automorphism of $X(n,t,q)$ is induced by a mapping on the points and lines of $\PG(n+t,q)$ which is incidence preserving, hence is a collineation.

Every automorphism $\psi$ of $X(n,t,q)$ is a permutation of the $(t-1)$-dimensional subspaces of
$\PG(n + t, q)$ disjoint from $\pi$.
Consider a $(t-2)$-dimensional space $\mu$ disjoint from $\pi$, and take two elements $\nu_1,\nu_2 \in {\P}$ such that $\nu_1 \cap \nu_2 =\mu$. We clearly have $\langle \nu_1, \nu_2 \rangle \in {\L}$.
Since $\psi$ preserves ${\L}$, $\psi$ sends $\nu_1$ and $\nu_2$ to two elements $\nu_1'$ and $\nu_2'$ of ${\P}$ lying in an element of ${\L}$, thus intersecting in a $(t-2)$-dimensional space $\mu'$.

Now consider $\nu_3 \in {\P}$ containing $\mu$, but not lying in $\langle \nu_1,\nu_2\rangle$. As seen before, its image $\nu_3'$ intersects both $\nu_1'$ and $\nu_2'$ in a $(t-2)$-dimensional space. Hence $\nu_3' \cap \nu_1' =\nu_3' \cap \nu_2' =\nu_1' \cap \nu_2'=\mu'$, since otherwise $\nu_1', \nu_2'$ and $\nu_3'$ would lie in one and the same element of ${\L}$, a contradiction.
It follows that $\psi$ extends to a well-defined mapping on the $(t-2)$-dimensional subspaces $\mu$ of $\PG(n + t, q)$ disjoint from $\pi$ by putting $\psi(\mu):=\psi(\nu_1)\cap \psi(\nu_2)$ for $\nu_1,\nu_2\in \P$ with $\nu_1\cap \nu_2=\mu$. Furthermore, $\psi$ preserves incidence between these $(t-2)$-spaces and the $(t-1)$-spaces of ${\P}$.

In other words, for $k=t-1$ the map $\psi$ extends to a mapping:
\begin{itemize}
\item[(i)] permuting $(k+1)$-subspaces intersecting $\pi$ in exactly one point, such that incidence with the $k$-subspaces disjoint from $\pi$ is preserved,
\item[(ii)] permuting $k$-subspaces disjoint from $\pi$,
\item[(iii)] permuting $(k-1)$-subspaces disjoint from $\pi$, such that incidence with the $k$-subspaces disjoint from $\pi$ is preserved and such that incidence with the $(k+1)$-subspaces intersecting $\pi$ in one point is preserved.
\end{itemize}

Now we continue by induction. Suppose that for some $k \leq t-1$, the previous three properties are valid. We will prove they also hold for $k-1$. Clearly, property (ii) for $k-1$ follows from property (iii) for $k$. We proceed in two steps, first proving (i) and afterwards (iii).
%\begin{itemize}
%\item permuting subspaces intersecting $\pi$ in exactly one point of dimension $k$, such that incidence with the $(k-1)$-dimensional subspaces disjoint from $\pi$ is preserved,
%\item permuting subspaces disjoint from $\pi$ of dimension $k-2$, such that incidence with the $(k-1)$-dimensional subspaces disjoint from $\pi$ is preserved and such that incidence with the $k$-dimensional subspaces intersecting $\pi$ in one point is preserved.
%\end{itemize}

Consider a $k$-dimensional space $\alpha$ intersecting $\pi$ in exactly one point $P$, and take two $(k+1)$-dimensional spaces $\beta_1, \beta_2$ intersecting $\pi$ only in $P$ such that $\beta_1 \cap \beta_2 =\alpha$. The map $\psi$ preserves the $(k-1)$-dimensional spaces of $\PG(n+t,q)$, disjoint from $\pi$, contained in $\alpha$ and their incidence with $\beta_1$ and $\beta_2$. The subspaces $\beta_1$ and $\beta_2 $ are mapped to two $(k+1)$-dimensional spaces $\beta_1'$ and $\beta_2'$, intersecting in a subspace $\alpha'$. The subspace $\a'$ needs to contain exactly $q^k$ $(k-1)$-dimensional spaces of $\PG(n+t,q)$ disjoint from $\pi$, hence $\alpha'$ is a $k$-dimensional subspace intersecting $\pi$ in a point $P'$. It follows that this point $P'$ is the point at infinity for both $\beta_1'$ and $\beta'_2$.

Consider a third $(k+1)$-dimensional space $\beta_3$, intersecting $\pi$ in exactly $P$, and containing $\alpha$. Its image $\beta_3'$, intersects both $\beta_1'$ and $\beta_2'$ in a $k$-dimensional space. Since $\psi$ acts bijectively on $(k-1)$-dimensional subspaces disjoint from $\pi$ and preserves their incidence with $(k+1)$-spaces intersecting $\pi$ in one point, the intersections $\beta_3' \cap \beta_1'$, $\beta_3' \cap \beta_2'$ and $\beta_1' \cap \beta_2'$ all contain the same $(k-1)$-dimensional subspaces and, again by counting, we see that they each coincide with the same $k$-dimensional subspace $\alpha'$ intersecting $\pi$ in $P'$.
It follows that $\psi$ extends to a well-defined mapping on the $k$-dimensional subspaces of $\PG(n + t, q)$ intersecting $\pi$ in one point, preserving incidence with the $(k-1)$-dimensional subspaces disjoint from $\pi$.

We now continue with the second step of the induction.
Consider a $(k-2)$-dimensional subspace $\gamma$ of $\PG(n+t,q)$ disjoint from $\pi$. Take two $(k-1)$-dimensional subspaces $\delta_1$ and $\delta_2$, disjoint from $\pi$, such that $\delta_1 \cap \delta_2 = \gamma$ and such that $\langle \delta_1, \delta_2 \rangle = \epsilon$ is a $k$-dimensional space intersecting $\pi$ in exactly one point $Q$.
From the previous part, we have that their images $\delta_1'$ and $\delta_2'$ are contained in $\epsilon'$, which is the image of $\epsilon$, hence they intersect in a $(k-2)$-dimensional subspace $\gamma'$, disjoint from $\pi$.

Consider a third $(k-1)$-dimensional subspace $\delta_3$ containing $\gamma$, disjoint from $\pi$, not contained in $\epsilon$, such that $\langle \delta'_1, \delta'_2, \delta'_3 \rangle$ is a $(k+1)$-dimensional space intersecting $\pi$ in the point $Q$.
From the above, the image $\delta_3'$ intersects both $\delta_1' $ and $\delta_2'$ in a $(k-2)$-dimensional space. The intersection $\delta_3' \cap \delta_1'$ equals $\delta_3' \cap \delta_2'$, since otherwise $\langle \delta_1, \delta_2, \delta_3 \rangle$ would be $k$-dimensional, a contradiction because $\psi$ preserves $(k+1)$-dimensional spaces intersecting $\pi$ in exactly in one point.

Hence, the map $\psi$ extends to a well-defined mapping on the $(k-2)$-dimensional subspaces of $\PG(n + t, q)$, disjoint from $\pi$, preserving incidence with the $k$-dimensional subspaces intersecting $\pi$ in one point, and with the $(k-1)$-dimensional subspaces disjoint from $\pi$.

By induction we see that the map $\psi$ extends to a well-defined incidence preserving mapping on the $m$-dimensional spaces, $0 \leq m \leq t-1$,  disjoint from $\pi$ and the $l$-dimensional spaces, $1 \leq l \leq t$, intersecting $\pi$ in exactly one point.

We now only need to show that $\psi$ can be extended to a bijective mapping on the points of $\pi$. Suppose that $P$ is a point of $\pi$, and $L_1$ and $L_2$ are two lines such that $L_1\cap \pi=L_2\cap \pi=\{P\}$. Since the plane $\langle L_1, L_2 \rangle$ is mapped to a plane intersecting $\pi$ in one point $P'$, both $L_1$ and $L_2$ are mapped to lines intersecting $\pi$ in $P'$. Consider a line $L_3$ intersecting $\pi$ in exactly $P$, such that $\langle L_1, L_3 \rangle$ intersects $\pi$ in a line. Then $\langle L_2, L_3 \rangle$ intersects $\pi$ only in $P$, hence also $L_3$ is mapped to a line intersecting $\pi$ in $P'$. So we can extend $\psi$ to a mapping on the points of $\pi$ by putting $\psi(P):=\psi(L_1)\cap \psi(L_2)$ for $L_1,L_2$ lines meeting $\pi$ exactly in the point $P$.

Now consider a line $L$ in $\pi$. Take a $3$-dimensional space $\mu$ intersecting $\pi$ in exactly $L$. Take two disjoint lines $M\neq L$ and $N\neq L$ in $\mu$; they are mapped to two disjoint lines spanning a $3$-dimensional space $\mu'$ intersecting $\pi$ in at most a line. The $q+1$ planes, spanned by $M$ and a point of $L$, all intersect $N$ in a point, hence their images lie in $\mu'$. The images of the points of $L$ are all different and lie in $\mu'$, hence they lie on a line $L'$ which has to be the intersection of $\mu'$ with $\pi$. It follows that $\psi$ also preserves lines of $\pi$.

We proved that every automorphism $\psi$ of $X(n,t,q)$ is induced by a mapping on the points and lines of $\PG(n+t,q)$ which is incidence preserving, hence is a collineation. From this, it is clear that the automorphism group of $X(n,t,q)$ is isomorphic to the stabiliser $\PGammaL(n+t+1,q)_{\pi}$ of $\pi$.
\end{proof}

%\begin{corollary}\label{aut}
%Since $T^*_n(\S) \cong X(n,t,q)$, we have $\Aut(T^*_n(\S))\cong \PGammaL(n+t+1,q)_{\pi}$.
%\end{corollary}

\section{$X(n,t,q)$  as a coset geometry}\label{sectionstefaan}

In this section we will see that $X(n,t,q)$ has a natural description as  a coset geometry. We will prove that $X(n,t,q)$ is isomorphic to the linear representation $T_n^*(\S)$ embedded in $\PG(n+1,q^t)$, where $\S$ is the subgeometry $\PG(n,q)$ of the hyperplane $H_\infty\cong\PG(n,q^t)$. Moreover, this will provide an elegant description of both the geometry and its automorphism group as determined in Section \ref{stabgeom}.

%\begin{remark}
%This case covers all possibilities of linear representations of non-trivial subgeometries, since a subgeometry $\PG(n,q^i)$ of $\PG(n,q^t)$ can be written as the subgeometry $\PG(n,\tilde{q})$ of $\PG(n,\tilde{q}^{\frac{t}{i}})$ where $\tilde{q}=q^i$ since $i|t$.
%\end{remark}

Without loss of generality, let $\pi$ be the $n$-dimensional space of $\PG(n+t,q)$ with equation $X_0=X_1=\ldots=X_{t-1}=0$.
We consider the embedding of the geometry $X(n,t,q)$ in $\PG(n+t,q)$. Recall that the point set $\P$ consists of the $(t-1)$-subspaces of $\PG(n+t,q)$ disjoint from $\pi$, and the line set $\L$ consists of the $t$-subspaces of $\PG(n+t,q)$ intersecting $\pi$ in exactly one point.

Consider the following $(n+1)$-dimensional spaces through $\pi$: $\Sigma_0: X_1=\ldots=X_{t-1}=0$, $\Sigma_j: X_0=\ldots=X_{j-1}=X_{j+1}=\ldots=X_{t-1}=0$, $j=1,\ldots,t-2$, $\Sigma_{t-1}: X_0=\ldots=X_{t-2}=0$. Every $(t-1)$-dimensional space $P \in \P$ is spanned by the set of $t$ unique points $\left\{ U_j=P \cap \Sigma_j \right\}_{j=0,\ldots,t-1}$, where $U_j$ has coordinates $(0, \ldots, 0, 1,0,\ldots,0,a_{0j},\ldots,a_{nj})$ with a $1$ at position $j$ and $a_{ij} \in \F_q$, for all $i=0,\ldots,n$, and $j=0,\ldots,t-1$.

For every $P \in \P$ we can consider a corresponding $(n+t+1)\times(n+t+1)$-matrix $A'_P$
\[A'_P=\left(\begin{matrix} I_t & 0 \\ A_P & I_{n+1} \end{matrix} \right),\] with the $(n+1) \times t$-matrix $A_P=(a_{ij})$, ${0\leq i \leq n, 0 \leq j \leq t-1} $, and where $I_k$ is the $k\times k$-identity matrix. Conversely every such matrix corresponds to a unique $P\in\P$.

Now let $G$ be the set consisting of all these matrices, \[G=\{A'_P \mid P \in \P\}.\] This set $G$, under operation of multiplication, forms a group. Note that
\[ A'_P A'_Q=\left(\begin{matrix} I_t & 0 \\ A_P + A_Q & I_{n+1} \end{matrix} \right) .\]
Hence, this group  is  elementary abelian of order $q^{(n+1)t}$. Furthermore, since every elementary abelian group corresponds to a vector space, this implies we will be able to interpret $X(n,t,q)$ as a geometry ``embedded'' in an $\F_p$-vector space, or equivalently, in an affine space over $\F_p$.

We define the action of $G$ on $\P$ by $A'_R(P)=Q$ if and only if $A'_RA'_P=A'_Q$. It is obvious that this is well defined and that the group $G$ acts sharply transitively on the point set $\P$ of $X(n,t,q)$.

One could also envision $G$ as a subgroup of $\PGL(n+t+1,q)$ and consider the action on the points of $\PG(n+t,q)$ by left-multiplication. The induced action on the elements of $\P$ is exactly the same as the one we defined earlier. Note that all the points of $\pi$ are fixed by the group $G$.

Next we will describe the lines of $X(n,t,q)$ in an algebraic way. Consider for every point $V \in \pi$ its coordinates $(0,\ldots,0,b_0,\ldots,b_n)$ (determined up to scalar multiple). We define the subgroup $G(V)$ of $G$ as follows:
\[G(V)=\left\{\left(\begin{matrix}  I_{t} & 0 \\ B_{\bar{a}} & I_{n+1} \end{matrix}\right) \mid \bar{a} \in \F_q^t \right\},\]
with the $(n+1)\times t$-matrix $B_{\bar{a}}=(b_i a_j )$, $0 \leq i \leq n, 0\leq j \leq t-1 $ for $\bar{a}=(a_0,a_1,\ldots,a_{t-1}) \in \F_q^t$. Clearly the group $G(V)$ is independent of the chosen coordinates for $V$.

For $0\leq i\leq n+t$, suppose the point $W_i$ of $\PG(n+t,q)$ has coordinates $w_i$, where $w_0=(1,0,\ldots,0),w_1=$ $(0,1,0,\ldots,0),\ldots,w_{n+t}=(0,\ldots,0,1)$.
The group $G(V)$ has size $q^t$ and stabilises the $t$-space $L_V=\langle V, W_0, W_1,\ldots,W_{t-1}\rangle$, note that this space is an element of $\L$. The group $G(V)$ fixes the point $V$ and acts transitively on the $t$-tuples $(V_0,\ldots,V_{t-1})$, where $V_i$ is a point of $\langle V, W_i\rangle\setminus\{V\}$.  Hence, it acts transitively on the $(t-1)$-spaces of $L_V$ not through $V$. These are exactly the elements of $\P$ contained in this line $L_V$ of $X(n,t,q)$. Furthermore, since this group has size $q^t$, this action is sharply transitive.

The space $I = \langle W_0, \ldots, W_{t-1} \rangle \in \P$ corresponds to the identity matrix of $G$. From the above, we learned that the lines of $X(n,t,q)$ through $I$ correspond to the subgroups  $G(V)$, $V\in\pi$. Furthermore, since $G$ (interpreted as a subgroup of $\PGL(n+t+1,q)$) fixes all points of $\pi$ and acts (sharply) transitively on the $(t-1)$-spaces disjoint from $\pi$, we easily deduce that the elements of $\L$ (the lines of $X(n,t,q)$) are in one to one correspondence with the cosets of $G(V)$ in $G$. However, at this point we can simplify notation if we take into account that all important properties of $G$ and the subgroups $G(V)$ are determined by the $(n+1)\times t$-submatrices in the lower left corner of the elements of $G$. Let ${\bf M}=M_q(n+1,t)$ be the group of all  $(n+1)\times t$-matrices over $\F_q$ under matrix addition.
We have obtained the following description of $X(n,t,q)$ as a coset geometry $\M=(\P_{\M},\L_{\M})$ with natural incidence (containment), point set $\P_{\M}$ and line set $\L_{\M}$ as follows:
\begin{itemize}
\item[$\P_{\M}$:] the elements of ${\bf M}$, that is, the $(n+1)\times t$-matrices over $\F_q$,
\item[$\L_{\M}$:] the cosets in ${\bf M}$ of the subgroups $L_{\bar{b}}:=\{\bar{b}^T\bar{a}\mid\bar{a}\in\F_q^t\}$, for all $ \bar{b} \in \F_q^{n+1}\setminus\{\overline{0}\}$.
\end{itemize}

There are exactly $\frac{q^n-1}{q-1}$ lines of type $L_{\bar{b}}$, since $L_{\bar{b}}=L_{\bar{c}}$, when $\bar{b}$ and $\bar{c}$ are scalar multiples of each other.

Note that this also provides a nice description of the point graph of our geometry as a Cayley graph: the vertices are the elements of ${\bf M}$, and two vertices are adjacent if and only if their difference (in ${\bf M}$) is of the form  $\bar{b}^T\bar{a}$ for some $\bar{a}\in\F_q^t$ and some $\bar{b} \in \F_q^{n+1}\setminus\{\overline{0}\}$.

Next we will find the lowest dimensional affine space in which $X(n,t,q)$ naturally embeds, and establish the isomorphism with $T_n^*(\S)$ in an algebraic way.
Let $f(x)=m_0+m_1x+\cdots+m_{t-1}x^{t-1}+x^t$ be an irreducible monic polynomial of degree $t$ over $\F_q$ used to construct $\F_{q^t}$. Let $M$ be the companion matrix of $f(x)$, that is,
\[M=\left(\begin{matrix}  0& I_{t-1} \\ -m_0 & -\overline{m} \end{matrix}\right), \]
with $\overline{m}=(m_1,\hdots,m_{t-1})$. Then it is well known from linear algebra (see for example \cite{Horn}, Theorem 3.3.14) that $f(x)$ is the minimal polynomial of $M$. Consequently, if we define
\[H=\{a_0I_t+a_1M+\cdots+a_{t-1}M^{t-1}\mid a_i\in\F_q\},\]
then $H$ has the structure of $\F_{q^t}$ under usual matrix addition and multiplication (this construction for example also appears in \cite{Lidl}). It  follows that $H\setminus\{0\}$ acts sharply transitively on the points of $\AG(t,q)$ different from $(0,\ldots,0)$.

Now define an action of $H$ on ${\bf M}$ in the following way:
\[(H,{\bf M})\rightarrow {\bf M}: (C,A_P)\mapsto A_PC.\]
It is readily checked that this makes ${\bf M}$ into an $H$-vector space, that is, an $\F_{q^t}$-vector space. However, if we consider this action restricted to the subgroups $L_{\bar{b}}$ of ${\bf M}$, then we see that we in fact obtain an action
\[(H,L_{\bar{b}})\rightarrow L_{\bar{b}}.\] This makes the $L_{\bar{b}}$ into $H$-vector subspaces of the $H$-vector space ${\bf M}$.

From the above it now follows that we can view our geometry $X(n,t,q)$ as a geometry embedded in an $(n+1)$-dimensional vector space over $\F_{q^t}$, where the lines through $I$ correspond to certain vector subspaces and the other lines to cosets of these subspaces, which are parallel subspaces when seen in $\AG(n,q^t)$. %Hence we obtain a representation of $X(n,t,q)$ as a generalised linear representation. 
Since the lines of $X(n,t,q)$ have size $q^t=\left|H\right|=\left|\F_{q^t}\right|$, these subspaces are one-dimensional and we obtain a linear representation of $X(n,t,q)$ as a point set in $\PG(n,q^t)$.

Let $q=p^h$, $p$ prime. Define the following \[\A=\{(A,B,C,l)\mid A\in{\bf M}, B\in\GL(n+1,q), C\in\GL(t,q), l\in\mathbb{Z}_h\},\]
and a binary operation $\circ$ on $\A$ as follows \[(A_2,B_2,C_2,l_2)\circ(A_1,B_1,C_1,l_1) = (B_2^{p^{-l_1}}A_1C_2^{p^{-l_1}}+A_2^{p^{-l_1}},B_2^{p^{-l_1}}B_1, C_1C_2^{p^{-l_1}},l_1+l_2).\]
Then $\A,_\circ$ is easily checked to be a group.

Next define an action of $\A$ on the points of $\M$ as follows \[ ((A,B,C,l), A_P)\mapsto \left(BA_PC+A\right)^{p^l}.\]
A simple verification now shows that this makes $\A$ into a group of automorphisms of $\M$.
The kernel of the described action clearly is $K=\{(0,\lambda I_{n+1},\lambda^{-1} I_t,0)\mid \lambda\in \F_q^*\}$.

Elements of the form $(A,I_{n+1},I_t,0)$ map every element of $\L_{\M}$ to one of its cosets, meaning it fixes the point set at infinity of the linear representation while permuting the lines that go through a given point at infinity.
An element $(0,I_{n+1},C,0)$ fixes every subgroup $L_{\bar{b}}$ and permutes the group elements of every subgroup, meaning it fixes the point set at infinity and the affine point corresponding to the zero matrix (meaning the space $I$), while permuting the points of every line of $\P$ through this point. Elements of the form $(0,I_{n+1},I_t, l)$ provide the semi-linear maps corresponding to the elements of $\Aut(\F_q)$.
The action of the subgroup $\{(0,B,I_t,0)\mid B\in\GL(n+1,q)\}\leq\A$ fixes the point corresponding to the zero matrix (that is, the origin in the corresponding linear representation) while permuting the lines that go through it. Hence this subgroup stabilises the point set at infinity of the linear representation.

We still need to uncover this point set at infinity. With every vector $\bar{b} \in \F_q^{n+1}\setminus\{\overline{0}\}$, up to scalar multiple, there is a corresponding point at infinity. Define a point-line incidence structure with as points the subgroups $L_{\bar{b}}$ and as lines the set of subgroups $\{L_{\bar{b_i}}|i=0,\ldots,q\}$ where the vectors $\bar{b_i}$ lie in a plane, meaning, the corresponding projective points lie on a projective line. This means the lines through the origin of the corresponding linear representation correspond naturally to the structure $\PG(n,q)$. Since this structure is stabilised by the subgroup $\{(0,B,I_t,0)\mid B\in\GL(n+1,q)\}$ of the automorphism group of $\M$, we know that our point set at infinity is a subgeometry isomorphic to $\PG(n,q)$ and hence that $\M\cong T_n^*(\S)$.

Finally, since this group has the size of the full automorphism group
\[\frac{\left|\A\right|}{\left| K\right|}=q^{(n+1)t}\left|\GL(n+1,q)\right|\left|\GL(t,q)\right|\left|\Aut(\F_q)\right|/(q-1)=\left|\PGammaL(n+t+1,q)_{\pi}\right|,\]

we see that $\A$ provides a natural description of the full automorphism group of $T_n^*(\S)$.
%Note that the geometric automorphism group can be described by the elements of $\A$ of the form $(A,B,I_t,l)$.

From this section we arrive to the following conclusion.

\begin{theorem}\label{aut}
The geometries $X(n,t,q)$ and $T^*_n(\S)$ are isomorphic and $\Aut(T^*_n(\S)) \cong \Aut(X(n,t,q)) \cong \PGammaL(n+t+1,q)_{\pi}$.
\end{theorem}

\section{The linear representation and its automorphism group}\label{sectionblowup}
%We have considered the linear representation $T^*_{n}(\S)$ of $\S\cong\PG(n,q)$, where $\S$ is a subgeometry $\PG(n,q)$ of $H_\infty\cong\PG(n,q^t)$. In Section \ref{sectionstefaan} we have proven that $T^*_n(\S)$ is isomorphic to $X(n,t,q)$ and hence it follows that $\Aut(T_n^*(\S)) \cong \Aut(X(n,t,q))$.
In this section, by some easy counting, we will show that while the automorphism group of $T_n^*(\S)$ is not induced by collineations of its ambient space, there is yet another setting where we do see that the automorphisms are induced by collineations of the ambient space, namely when we consider the generalised linear representation isomorphic to $T_n^*(\S)$.

\begin{definition} Let $\K$ be a set of disjoint $(t-1)$-dimensional subspaces in $\Pi_\infty\cong\PG(m,q)$, $q$ a prime power. Embed $\Pi_\infty$ as a hyperplane in $\PG(m+1,q)$. The {\em generalised linear representation} $T_{m,t-1}^*(\K)$ of $\K$ is the incidence structure $(\P', \L')$ with natural incidence and with:
\begin{itemize}%\setlength{\itemindent}{3.5em}
\item[$\P':$] the {\em affine} points of $\PG(m+1,q)$, i.e. the points of $\PG(m+1,q)\setminus \Pi_\infty$,
\item[$\L':$] the $t$-dimensional spaces of $\PG(m+1,q)$ containing a $(t-1)$-space of $\K$, but not lying in $\Pi_\infty$.
\end{itemize}
\end{definition}

When $t=1$, clearly this definition coincides with the definition of a linear representation.

By {\em field reduction}, the points of $\PG(n, q^t)$ correspond to the elements of a Desarguesian $(t - 1)$-spread $\D_{\infty}$ of $J_\infty\cong\PG(t(n+1)-1, q)$. We will denote the element of $\D_{\infty}$ corresponding to a point $P$ of $\PG(n, q^t)$ by $\mathcal{F}(P)$, and define $\mathcal{F}(\pi):=\{\mathcal{F}(P)\vert P\in \pi\}$ for a subset $\pi$ of $\PG(n,q^t)$.

If $\S$ is a subgeometry $\PG(n,q)$ of $H_{\infty}$, then the set $\mathcal{F}(\S)$ is a set of $\frac{q^{n+1}-1}{q-1}$ disjoint $(t-1)$-spaces in $J_{\infty}$. It is well-known that this set forms one of the two {\em systems} of the Segre variety $S_{n,t-1}$ in $\PG(t(n+1)-1,q)$.
%If $U$ is a subset of $\PG(t(n+1)-1, q)$, then we define $\B(U) := \{R \in \D \mid U\cap R \neq \emptyset\}$. We can identify the elements of $\B(U)$ with their corresponding points of $\PG(n, q^t)$, i.e. we will not always make the distinction between a point $P$ and its corresponding spread element $\mathcal{F}(P)$.
%
%Now the subgeometry $\S$ corresponds to $\B(\pi)$, where $\pi$ is an $n$-dimensional space of $J_\infty$; in fact, there are exactly $\frac{q^t-1}{q-1}$ spaces $\pi$ for which $\mathcal{F}(\S)=\B(\pi)$, namely every $n$-space of the second system of $S_{n,t-1}$.
For more information on field reduction and this Segre variety we refer to \cite{MichelGeertruiFieldReduction}.

Let us now recall the {\em representation of Barlotti-Cofman} \cite{Barlotti} of $\PG(n+1,q^t)$ inside $\PG(t(n+1),q)$. Here, the points of the hyperplane $H_\infty\cong\PG(n,q^t)$ in $\PG(n+1,q^t)$ are represented as $(t-1)$-dimensional spaces of a Desarguesian spread $\D_\infty$ in $J_\infty\cong\PG(t(n+1)-1,q)$. The affine points of $\PG(n+1,q^t)$ with respect to $H_\infty$ can be identified with the affine points of a projective space $\PG(t(n+1),q)$, with respect to the hyperplane $J_\infty$. The lines of $\PG(n+1,q^t)$ intersecting $H_\infty$ in a point correspond to the $t$-dimensional spaces of $\PG(t(n+1),q)$, meeting $J_\infty$ in an element of $\D_\infty$. %As seen before, there exists a $n$-dimensional subspace $\pi$ in $J_\infty$ such that the elements of $\mathcal{F}(\S)$ correspond to the points of $\S$.

We consider the Barlotti-Cofman representation of the points and lines of the linear representation $T^*_n(\S)$. This means that the points of $\S$ in the hyperplane $H_\infty$ are represented as the set of elements $\F(\S)$ in the Desarguesian $(t-1)$-spread $\D_\infty$, the affine points of $\PG(n+1,q^t)$ correspond to affine points of $\PG(t(n+1),q)$ and the lines of $T_n^*(\S)$ correspond to $t$-spaces, having an element of $\F(\S)$ at infinity. In this way we get the generalised linear representation $T_{t(n+1)-1,t-1}^*(\mathcal{F}(\S))$, which is then clearly isomorphic to $T^*_n(\S)$.

It is clear that the group $\PGammaL(t(n+1)+1,q)_{\mathcal{F}(\S)}$ stabilises the generalised linear representation $T_{t(n+1)-1,t-1}^*(\mathcal{F}(\S))$ and hence is isomorphic to a subgroup of $\Aut(X(n,t,q)) \cong \PGammaL(n+t+1,q)_{\pi}$. We will see by counting that the groups are in fact isomorphic.

\begin{definition}
For a hyperplane $H$ of $\PG(m,q)$ we denote the set of elements of $\PGammaL(m+1,q)$ fixing all points of
the hyperplane $H$ as $\Persp_q(H)$. Note that this is in fact a subgroup of $\PGL(m+1,q)$. Clearly, the group $\Persp_q(H)$ %consists of all elations and homologies with axis $H$, and
has size $|\Persp_q(H)|=q^{m}(q-1)$.
\end{definition}
\begin{result}{\rm \label{ext} \cite{wij2}} Suppose $\K$ is a subset of a hyperplane $H$ of $\PG(m,q)$ such that $\langle \K \rangle = H$. The group $\PGammaL(m+1,q)_\K$ is an extension of $\Persp_q(H)$ by $\PGammaL(m,q)_\K$ and $\PGL(m+1,q)_\K$ is an extension of $\Persp_q(H)$ by $\PGL(m,q)_\K$.
\end{result}

\begin{result}{\rm \label{stabsegre}\cite[Theorem 25.5.13]{GGG}}
The projective automorphism group of a Segre variety $S_{l,k}$ of $\PG((l+1)(k+1)-1, q)$ is either
isomorphic to $\PGL(l + 1, q) \times \PGL(k + 1, q)$ if $l \neq k$ or is isomorphic to $(\PGL(l +1,q) \times \PGL(k + 1, q)) \rtimes C_2$ if $l = k$.
\end{result}

To denote the number of $(k-1)$-dimensional subspaces of $\PG(m-1,q)$, we will use the following notation: \[\left[\begin{matrix} m \\ k \end{matrix} \right]_q=\frac{(q^m-1)(q^{m-1}-1)\ldots(q^{m-k+1}-1)}{(q^k-1)(q^{k-1}-1)\ldots(q-1)}.\]
\begin{theorem}
$\Aut(T^*_n(\S)) \cong \PGammaL(n+t+1,q)_{\pi} \cong \PGammaL(t(n+1)+1,q)_{\mathcal{F}(\S)}$
\end{theorem}
\begin{proof}
The full automorphism group of $T^*_n(\S)$ is isomorphic to $\PGammaL(n+t+1,q)_{\pi}$, see Theorem \ref{aut}. Since $\PGammaL(n+t+1,q)$ acts transitively on the $n$-spaces of $\PG(n+t,q)$, we find the following:
\begin{align*}
|\PGammaL(n+t+1,q)_{\pi}| &= \frac{|\PGammaL(n+t+1,q)|}{\left[\begin{matrix} n+t+1 \\ n+1 \end{matrix} \right]_q}\\
&= q^{t(n+1)}q^{\frac{t(t-1)}{2}}(q^t-1)\ldots(q-1)|\PGammaL(n+1,q)|.
\end{align*}
Now we calculate the size of $\PGammaL(t(n+1)+1,q)_{\mathcal{F}(\S)}$. By Result \ref{ext}, since $\mathcal{F}(\S)$ spans $J_\infty$, we find:
\begin{align*}
|\PGammaL(t(n+1)+1,q)_{\mathcal{F}(\S)}| &= |\Persp_q(J_\infty)||\PGammaL(t(n+1),q)_{\mathcal{F}(\S)}|. \\
\end{align*}%
As seen before, the set of points contained in $\mathcal{F}(\S)$ forms a Segre variety $S_{n,t-1}$.   Hence the stabiliser of $\mathcal{F}(\S)$ is the stabiliser of the Segre variety  that in the case $t=n+1$ does not switch the two systems. Thus, by Result \ref{stabsegre} we find $\PGL(t(n+1),q)_{\mathcal{F}(\S)}= \PGL(n +1,q) \times \PGL(t, q)$. The semilinear automorphisms stabilising any one of the systems of the Segre variety naturally extend to elements of $\PGammaL(t(n+1),q)$. Hence $\left|\PGammaL(t(n+1),q)_{\mathcal{F}(\S)}\right|= \left|\PGL(n +1,q)\right| \left|\PGL(t, q)\right|\left|\Aut(\F_q)\right|$.

We conclude that $\PGammaL(t(n+1)+1,q)_{\mathcal{F}(\S)}$ has the same size as $\PGammaL(n+t+1,q)_{\pi}$:
\begin{align*}
|\PGammaL(t(n+1)+1,q)_{\mathcal{F}(\S)}| &= |\Persp_q(J_\infty)||\PGL(n+1,q)||\PGL(t,q)|\left|\Aut(\F_q)\right| \\
&=|\Persp_q(J_\infty)||\PGammaL(n+1,q)||\PGL(t,q)| \\
&= q^{t(n+1)}(q-1)q^{\frac{t(t-1)}{2}}(q^t-1)\ldots(q^{2}-1)|\PGammaL(n+1,q)|.
\end{align*}

Clearly every collineation of $\PGammaL(t(n+1)+1,q)_{\mathcal{F}(\S)}$ is a non-trivial element of $\Aut(T^*_n(\S))$. Both groups have the same size, so $\Aut(T^*_n(\S)) \cong \PGammaL(t(n+1)+1,q)_{\mathcal{F}(\S)}$.
\end{proof}

The subgroup of the automorphism group of the linear representation $T^*_n(\S)$ for which the elements are induced by collineations of the space $\PG(n+1,q^t)$ will be called the {\em geometric automorphism group}. Since $\S$ spans $H_\infty$, this group is isomorphic to $\PGammaL(n+2,q^t)_\S$.

\begin{theorem}
The full automorphism group of $T^*_n(\S)$ is $\frac{1}{t}q^{\frac{t(t-1)}{2}}(q^{t-1}-1)\ldots(q^2-1)(q-1)$ times larger than the geometric automorphism group of $T^*_n(\S)$.
\end{theorem}
\begin{proof}

The geometric automorphism group of $T^*_n(\S)$ is isomorphic to $\PGammaL(n+2,q^t)_{\S}$ and has the following size:
$$
|\PGammaL(n+2,q^t)_{\S}| = |\Persp_{q^t}(H_\infty)| |\Aut(\F_{q^t}) / \Aut(\F_q)| |\PGammaL(n+1,q)|$$
$$=q^{t(n+1)}(q^t-1)t|\PGammaL(n+1,q)|.$$

\vspace{-0.5cm}
\end{proof}

We end with a full answer to the isomorphism problem of linear representations $T_n^*(\K)$ and $T_n^*(\K')$ of point sets $\K$, $\K'$.

\begin{theorem} Let $\K$ and $\K'$ denote point sets in $H_\infty\cong\PG(n,q^t)$ such that the closure $\overline{\K}$ is a subgeometry $\PG(n,q)$ of $H_{\infty}$ and let $\alpha$ be an isomorphism between $T_n^*(\K)$ and $T_n^*(\K')$. Then $\alpha$ is induced by an element of $\PGammaL(t(n+1)+1,q)_{J_\infty}$ mapping $\FF(\K)$ onto $\FF(\K')$.
\end{theorem}

\noindent
Affiliation of the authors:\\

\noindent
sgdewint\makeatletter @mtu.edu\\
Department of Mathematics, Michigan Technological University\\
1400 Townsend Drive, Houghton, MI 49931, United States\\

\noindent
srottey\makeatletter @vub.ac.be\\
Department of Mathematics, Vrije Universiteit Brussel\\
Pleinlaan 2, 1050 Brussel, Belgium\\

\noindent
gvdevoorde\makeatletter @cage.ugent.be\\
Department of Mathematics, UGent\\
Krijgslaan 281-S22, 9000 Gent, Belgium\\

\end{document}